\documentclass[12pt]{amsart}
\usepackage{amssymb,latexsym,amsmath,amsthm,graphicx,epsfig,amssymb}
\usepackage{ifthen}
\usepackage{bm}
\usepackage{bbm}
\usepackage[all]{xy}
\usepackage{pgf}
\usepackage{tikz-cd}
\usepackage{color}

\newtheorem{theorem}{Theorem}[section]
\newtheorem{lemma}[theorem]{Lemma}
\newtheorem{definition}[theorem]{Definition}

\newtheorem{corollary}[theorem]{Corollary}
\newtheorem{proposition}[theorem]{Proposition}
\newtheorem{ex}[theorem]{Example}

\def\F{\mathcal{F}}

\def\M{\mathfrak{m}}

\def\N{\mathbb{N}}

\def\K{\mathbbm{k}}

\begin{document}
\title{\bf Betti splitting via componentwise linear ideals}
\author{Davide Bolognini}

\maketitle
\begin{abstract}
\noindent A monomial ideal $I$ admits a Betti splitting $I=J+K$ if the Betti numbers of $I$ can be determined in terms of the Betti numbers of the ideals $J,K$ and $J \cap K. $  Given a monomial ideal $I$, we prove that $I=J+K$ is a Betti splitting of $I$, provided $J$ and $K$ are componentwise linear, generalizing a result of Francisco, H\`a and Van Tuyl. If $I$ has a linear resolution, the converse also holds. We apply this result recursively to the Alexander dual of vertex-decomposable, shellable and constructible simplicial complexes and to determine the graded Betti numbers of the defining ideal of three general fat points in the projective space. 
\let\thefootnote\relax
\end{abstract}

\section{Introduction} %Mayer and Vietoris, 
Our aim is to pursue the spirit of Eliahou and Kervaire \cite{eiker}, Francisco,  H\`a and Van Tuyl \cite{ur} in order to find suitable decomposition of the Betti table of a monomial ideal, possibly available for recursive procedures.

Let $\K$ be a field and let $I \subseteq R=\K[x_1,...,x_n]$ be a monomial ideal. Consider $J,K \subseteq I$ monomial ideals such that the set of minimal monomial generators of $I$ is the disjoint union of the sets of minimal monomial generators of $J$ and $K$. We say that $I=J+K$ is a {\em Betti splitting} of $I$ if $$\beta_{i,j}(I)=\beta_{i,j}(J)+\beta_{i,j}(K)+\beta_{i-1,j}(J \cap K), \text{ for all } i,j \in \N,$$ where $\beta_{i,j}(-)$ denotes the graded Betti numbers of a minimal $R$-free graded resolution.

This approach was used for the first time by Eliahou and Kervaire \cite{eiker}, giving an explicit formula for the total Betti numbers of stable ideals. Fatabbi \cite{fat} developed the theory for graded Betti numbers. Many authors wrote papers applying the Eliahou-Kervaire technique to the resolution of special classes of monomial ideals (see e.g.\cite{far}, \cite{fat}, \cite{fran}, \cite{fam}, \cite{moradi}, \cite{anda}, \cite{tuyl}, \cite{tuyl2}, \cite{tito}). Francisco,  H\`a and Van Tuyl proved in \cite[Corollary 2.4]{ur} that if $J$ and $K$ have a linear resolution, then $I=J+K$ is a Betti splitting of $I$. In Section 3 we generalize this result, proving that $I=J+K$ is a Betti splitting of $I$, provided $J$ and $K$ are componentwise linear (Theorem \ref{my}). If $I$ has a linear resolution, the converse also holds (Proposition \ref{prequel}).

%If $I=J+K$ is a Betti splitting of $I$, there is a lower bound for the linearity defect of $I$ (see Section 2) in terms of linearity defects of $J$ and $K$ (Theorem \ref{ldbetti}). From this we prove that the converse of Theorem \ref{my} holds, under the additional assumption that $I$ is componentwise linear (Corollary \ref{converse}). 

Componentwise linear ideals have been extensively studied (see e.g.\cite{fran}, \cite{fam}, \cite{hibi},\cite{hibi2}).  Stable ideals, ideals with linear quotients and ideals with linear resolution are examples of componentwise linear ideals.

%Notice that, up to polarization, componentwise linear monomial ideals can be considered squarefree (Proposition \ref{polariz}) and the action preserves the numerical invariants of the minimal free resolution. 
In Section 4 we apply the theory of Betti splittings to the Alexander dual ideal $I_{\Delta}^*$ of a simplicial complex $\Delta$. By the Stanley-Reisner correspondence, squarefree monomial ideals correspond to simplicial complexes (see e.g.\cite{hibi}). This is an important bridge between Commutative Algebra and Combinatorics. In particular, by Hochster's formula \cite{hoch}, the graded Betti numbers of $I_{\Delta}^*$ reflect geometric and topological information on $\Delta$. 

%One of the main challenges in homology computation is to be able to deal with currently available data sets, thus leading to high-dimensional complexes with a large number of vertices. To reduce computational costs, the strategy is to decompose these shapes in smaller ones, make computations on pieces and then to recover the information about the original shape. 

%This gives to our paper one more motivation.%coming from shape recognition. Topological features of the objects can be captured by the study of simplicial homology. 

As a consequence of Theorem \ref{my}, we recover a result due to Moradi and Kosh-Ahang \cite{moradi}, proving that the Alexander dual of a vertex-decomposable simplicial complex admits a particular kind of splitting, the so-called $x_i$-splitting (Corollary \ref{vd}). In Corollary \ref{sh} we prove a Betti splitting result for shellable and constructible simplicial complexes, showing that in general they do not admit $x_i$-splitting (Example \ref{shnobordo}).

A further application is an extension of a result proved by Valla in \cite{tito}. By using a recursive approach we can compute explicitly the graded Betti numbers of the defining ideal of three general fat points in the projective space (Corollary \ref{finalcor} and Corollary \ref{finalcor2}).\\

{\bf Acknowledgements.}
The results presented in this paper are part of my PhD thesis and I would like to thank my advisors, Maria Evelina Rossi, Leila De Floriani and Emanuela De Negri for their encouragements and their constant care about my work. The author is grateful to Hop Nguyen for pointing out \cite[Lemma 2.8(ii)]{hop}. It gives substantial improvements to a previous version of Theorem \ref{my}. Thanks to Aldo Conca and Matteo Varbaro for helpful discussions and suggestions. 

\section{Preliminaries}\label{pre}
Let $\K$ be a field, $R=\K[x_1,...,x_n],$ $\M$ the maximal homogeneous ideal of $R,$ $I \subseteq R$ an homogeneous ideal. Denote by $\beta_{ij}(I)=\dim_{\K}\mathrm{Tor}^R_i(I,\K)_j$ the {\em graded Betti numbers} of $I$. In the following we omit the superscript $R$. Given a monomial ideal $I$, let $G(I)$ be the minimal system of monomial generators of $I$ and $\mathrm{indeg}(I)$ be the lowest degree of a generator in $G(I).$ 

%  and by $\beta_i(I)=\sum_{j \in \N}\beta_{i,j}(I)$ the total i-th Betti numbers of $I$ $deg(m)$ the degree of a monomial $m$, by . 

\begin{definition} \em{(Francisco, H\`a, Van Tuyl,\cite{ur})}
Let $I$, $J$ and $K$ be monomial ideals such that $I=J+K$ and $G(I)$ is the disjoint union of $G(J)$ and $G(K)$. Then $J+K$ is a {\em Betti splitting} of $I$ if $$\beta_{i,j}(I)=\beta_{i,j}(J)+\beta_{i,j}(K)+\beta_{i-1,j}(J \cap K), \text{ for all } i,j \in \N.$$
\end{definition} 

The previous definition can be given in terms of the vanishing of some $\mathrm{Tor}$ modules maps, as stated in the following result.

%Tor_{i+1}(\K,J \cap K)_j \rightarrow Tor_{i+1}(\K,J)_{j} \oplus Tor_{i+1}(\K,K)_{j} \rightarrow
% \rightarrow Tor_{i-1}(\K,J)_j \oplus Tor_{i-1}(\K,K)_j \rightarrow Tor_{i-1}(\K,J \cap K)_{j} \rightarrow

\begin{proposition}\label{formal}{\em (Francisco, H\`a and Van Tuyl, \cite[Proposition 2.1]{ur})}
Let $I$, $J$ and $K$ be monomial ideals such that $I=J+K$ and $G(I)$ is the disjoint union of $G(J)$ and $G(K)$. Consider the short exact sequence $$0 \rightarrow J \cap K \rightarrow J \oplus K \rightarrow I \rightarrow 0$$
and the corresponding long exact sequence of $\mathrm{Tor}$ modules:

\begin{center}
$\cdots \rightarrow \mathrm{Tor}_{i+1}(I,\K)_{j} \rightarrow \mathrm{Tor}_i(J \cap K,\K)_{j} \stackrel{\phi_{i,j}}\longrightarrow \mathrm{Tor}_i(J,\K)_{j} \oplus \mathrm{Tor}_i(K,\K)_{j} \rightarrow \mathrm{Tor}_i(I,\K)_{j} \rightarrow \cdots$
\end{center}

Then the following are equivalent:
\begin{itemize}
 \item[$(i)$] $I=J+K$ is a Betti splitting; 
 \item[$(ii)$] $\phi_{i,j}=0$, for all $i,j \in \N$.
\end{itemize}
\end{proposition}

If $I$ is generated in degree $d$ we say that $I$ has a {\em $d$-linear resolution} if $\beta_{i,i+j}(I)=0$ for every $i,j \in \N$ and $j \neq d$. When the context is clear, we simply write that $I$ has a linear resolution. The {\em Castelnuovo-Mumford regularity} of $I$ is defined by $\mathrm{reg}(I)=max\{j-i|\beta_{i,j}(I) \neq 0\}.$ An ideal $I$ generated in degree $d$ has a $d$-linear resolution if and only if $\mathrm{reg}(I)=d$.
\vskip2mm
Let $I,J$ and $K$ be monomial ideals, with $I=J+K$ and $G(I)$ is the disjoint union of $G(J)$ and $G(K)$. Francisco,  H\`a and Van Tuyl proved in \cite[Corollary 2.4]{ur} that if $J$ and $K$ have a linear resolution, then $I=J+K$ is a Betti splitting. In Section 3 we generalize this result assuming $J$ and $K$ componentwise linear.
\vskip1mm
We recall some definitions and results that will be useful later. These hold also in a more general setting, but from now on we assume $R$ to be the standard graded polynomial ring with coefficients in a field $\K$ (see \cite{yen},\cite{hop}, \cite{romer} and \cite{sega} for more details). %\cite{hop2},
\vskip 1mm
Componentwise linear ideals have been introduced by Herzog and Hibi in \cite{hibi2}. Denote by $I_{<j>}$ the ideal generated by all the homogeneous polynomials of degree $j$ belonging to $I$. In the monomial case, $I_{<j>}$ is simply the ideal generated by all monomials of degree $j$ belonging to $I$. 

\begin{definition}
\em A homogeneous ideal $I \subseteq R$ is called {\em componentwise linear} if $I_{<j>}$ has a linear resolution, for every $j \in \N$. 
\end{definition}
 
%This condition is actually a characterization of the Betti splitting, provided $I$ has a linear resolution (see Corollary \ref{converse}).
%%If $I$ is a monomial squarefree ideal, denote by $I_{[j]}$ the ideal generated by squarefree monomials of degree $j$ belonging to $I$ (for $j>n$ we have $I_{[j]}=0$). Note that $I$ is a componentwise linear ideal if and only if $I_{[j]}$ has a linear resolution, for each $j \in \N$ \cite[Proposition 1.5]{hibi2}. %The resolution of a monomial squarefree componentwise linear ideal $I$ can be computed from the resolutions of its graded pieces, as stated in the following result.

%\begin{proposition}\label{hh}(\cite[Corollary 1.6]{hibi2})
%Let $I$ be a monomial squarefree componentwise linear ideal. Then $(\M I)_{[j]}$ has a linear resolution for each $i,j \in \N$ we have $\beta_{i,i+j}(I)=\beta_i(I_{[j]})-\beta_i((\M I)_{[j]}).$
%\end{proposition}

Notice that we cannot detect if an ideal is componentwise linear from its graded Betti numbers. In the following example we show an ideal $I$ that is not componentwise linear but with the same graded Betti numbers of a componentwise linear ideal.

\begin{ex}{\em (\cite[Example 5.5]{hsv})}
\em Let $I,J \subseteq \K[x,y,z]$ be the ideals $I=(x^4,x^3y,x^2y^2,$ $x^3z,xy^2z,xyz^2,xy^4,x^2z^3,y^4z)$ and $J=(x^4,x^3y,x^2y^2,xy^3,y^4,x^3z,x^2yz^2,x^2z^3,xy^2z^2)$. Using CoCoA \cite{cocoa}, one can see that $I$ and $J$ have the same graded Betti numbers. The ideal $J$ is stable hence componentwise linear, while $I_{<4>}$ has not a linear resolution, thus $I$ is not componentwise linear. %$J$ is strongly stable and
\end{ex} 

%But first we need to recall some useful definitions and facts.\\% To give a useful characterization of componentwise linear ideals we need to recall some definitions and facts. Definitions and results in this section

Linearity defect was introduced by Herzog and Iyengar in \cite{yen} and measures how far a resolution is from being linear. 

Let $M$ be a finitely generated graded $R$-module and $\mathbb{F}$ its minimal graded free resolution over $R$. Let $\mathrm{lin}^R(\mathbb{F})$ be the chain complex obtained by $\mathbb{F}$ replacing by zero each entry of degree greater than one in the matrices of the differentials of $\mathbb{F}$.

\begin{definition}
\em Let $M$ be a finitely generated graded $R$-module $M$. The {\em linearity defect} of $M$ is defined by $$\mathrm{ld}_R(M):=\mathrm{sup}\{i:H_i(\mathrm{lin}^R(\mathbb{F})) \neq 0\},$$ where $H_i(\mathrm{lin}^R(\mathbb{F}))$ denotes the $i$-th homology of the chain complex $\mathrm{lin}^R(\mathbb{F}).$ 
\end{definition}

%Modules with linearity defect zero are called {\em Kozsul modules}. %We recall here a useful characterization of linearity defect in terms of $\mathrm{Tor}$ modules.

%\begin{theorem}\label{segaresult}\em(\c{S}ega,\cite[Theorem 2.2]{sega})
%\em For any finitely generated $R$-module $M \neq 0$ the following are equivalent:
%\begin{itemize}
%\item $\mathrm{ld}_R(M) \leq k$;
%\item the natural map $\mathrm{Tor}_i(M,R/\M^{s+1}) \rightarrow \mathrm{Tor}_i(M,R/\M^s)$ is zero for all $i>k$, $s \geq 0$.
%\end{itemize}
%In other words $$\mathrm{ld}_R(M)=\inf\{k:\mathrm{Tor}_i(M,R/\M^{s+1}) \rightarrow \mathrm{Tor}_i(M,R/\M^s) \text{  is zero, for all } i>k \text{ and } s \geq 0\}.$$
%\end{theorem}

We close this section by stating a special case of a useful characterization of componentwise linear modules due to R\"{o}mer (see also \cite[Proposition 4.9]{ya}).

\begin{theorem}\label{romerresult}\em(R\"{o}mer,\cite[Theorem 3.2.8]{romer})
\em For any homogeneous ideal $I \subseteq R$, $I$ is componentwise linear if and only if $\mathrm{ld}_R(I)=0$.
\end{theorem}

%Then an homogeneous ideal $I$ over the polynomial ring $R$ is componentwise linear if and only if $\mathrm{ld}_R(I)=0$. If moreover $I$ is generated in one degree it has linear resolution.

\section{Results}\label{main}
Let $I,J$ and $K$ be monomial ideals, with $I=J+K$ and $G(I)$ is the disjoint union of $G(J)$ and $G(K)$. Francisco,  H\`a and Van Tuyl proved in \cite[Corollary 2.4]{ur} that if $J$ and $K$ have a linear resolution, then $I=J+K$ is a Betti splitting. Provided $I$ with a linear resolution, the converse also holds, as stated in the following result.

\begin{proposition}\label{prequel}
Let $d$ be a positive integer, $I$ be a monomial ideal with a $d$-linear resolution, $J,K \neq 0$ monomial ideals such that $I=J+K$, $G(I)=G(J) \cup G(K)$ and $G(J) \cap G(K)=\emptyset$. Then the following facts are equivalent:
\begin{enumerate}
 \item[$(i)$] $I=J+K$ is a Betti splitting of $I$;
 \item[$(ii)$] $J$ and $K$ have $d$-linear resolutions.
\end{enumerate}
If this is the case, then $J \cap K$ has a $(d+1)$-linear resolution.
\end{proposition}

\begin{proof} Assume $(i)$ holds. Then $\beta_{i,i+j}(I)=\beta_{i,i+j}(J)+\beta_{i,i+j}(K)+\beta_{i-1,i+j}(J \cap K)$ for all $i,j \geq 0$. Let $i \geq 0$. For $j \neq d$ we have $\beta_{i,i+j}(I)=\beta_{i,i+j}(J)=\beta_{i,i+j}(K)=0, $ thus $J$ and $K$ have a $d$-linear resolution. 

Conversely one has that $(ii)$ implies $(i)$, by \cite[Corollary 2.4]{ur}.

By \cite[Corollary 2.2]{ur} we have $\mathrm{reg}(J \cap K) \leq \mathrm{reg}(I)+1=d+1$. Since $\mathrm{indeg}(J \cap K) \geq d+1$, then $\mathrm{reg}(J \cap K) \geq \mathrm{indeg}(J \cap K) \geq d+1$, thus $J \cap K$ has a $(d+1)$-linear resolution.  \end{proof}

Notice that there exist ideals with a linear resolution that do not admit any Betti splitting (see Example \ref{nosplit}).

Now we extend \cite[Corollary 2.4]{ur}, assuming $J$ and $K$ componentwise linear ideals. In the proof we use the following key-lemma.

\begin{lemma}\label{hop1}\em(Nguyen,\cite[Lemma 2.8(ii)]{hop})
\em Let $M \rightarrow P$ be an $R$-linear map between finitely generated $R$-modules. If for some $k \geq \mathrm{ld}_R(P)+1$, the map $\mathrm{Tor}_{k-1}(M,\K) \rightarrow \mathrm{Tor}_{k-1}(P,\K)$ is zero, then the map
$$\mathrm{Tor}_i(M,R/\M^s) \rightarrow \mathrm{Tor}_i(P,R/\M^s)$$ is zero for all $i \geq k$ and all $s \geq 0$.
\end{lemma}

\begin{theorem}\label{my}
Let $I,J$ and $K$ be monomial ideals such that $I=J+K$ and $G(I)$ is the disjoint union of $G(J)$ and $G(K)$. If $J$ and $K$ are componentwise linear, then $I=J+K$ is a Betti splitting of $I$.
\end{theorem}

\noindent {\em Proof.} By Proposition \ref{formal}, it suffices to prove that all the maps $$\mathrm{Tor}_i(J \cap K,\K) \stackrel{\phi_{i}}\longrightarrow \mathrm{Tor}_i(J \oplus K,\K)=\mathrm{Tor}_i(J,\K) \oplus \mathrm{Tor}_i(K,\K)$$ are zero for all $i \geq 0$. Since $J$ and $K$ are componentwise linear, by Theorem \ref{romerresult} we have $\mathrm{ld}_R(J)=\mathrm{ld}_R(K)=0$. Then $\mathrm{ld}_R(J \oplus K)=0$. Since $G(I)$ is the disjoint union of $G(J)$ and $G(K)$, hence $J \cap K \subseteq \M J$ and $J \cap K \subseteq \M K$. Then the map $$\mathrm{Tor}_0(J \cap K,\K) \stackrel{\phi_{0}}\longrightarrow \mathrm{Tor}_0(J \oplus K,\K)$$ is zero. By Lemma \ref{hop1} with $k=s=1$ the result follows. \hfill $\qed$\\

%From this, it follows $\bm{\phi_{i,j}}=0$ for all $j \geq 0$, as desired. , by Theorem \ref{segaresult}, Theorem \ref{my} cannot be improved in general, without further assumptions.

In view of Proposition \ref{prequel}, the assumptions on $J$ and $K$ in Theorem \ref{my} cannot be weakened in general, without further assumptions.\\

%The converse of Theorem \ref{my} holds provided $I$ and $J \cap K$ componentwise linear (see Corollary \ref{converse}).

%\begin{corollary}\label{converse}
%Let $I$ be a componentwise linear monomial ideal, $J,K \neq 0$ monomial ideals such that $I=J+K$ and $G(I)$ is the disjoint union of $G(J)$ and $G(K)$. Assume $J \cap K$ componentwise linear. Then the following facts are equivalent:
%\begin{enumerate}
%\item $I=J+K$ is a Betti splitting of $I$;
%\item $J$ and $K$ are componentwise linear.
%\end{enumerate}
%If this is the case, then $J \cap K$ is componentwise linear.
%end{corollary}

Clearly the converse of Theorem \ref{my} does not hold in general, as shown in the following example.

\begin{ex}
\em Let $I \subseteq \K[x_1,x_2,x_3,x_4,x_5,x_6]$ be the monomial ideal defined by $$I=(x_4x_5x_6,x_1x_2x_6,x_1x_3x_4).$$ Define $J=(x_4x_5x_6)$ and $K=(x_1x_2x_6,x_1x_3x_4)$. It is easy to check that $I=J+K$ is a Betti splitting of $I$. In fact the graded Betti numbers of $I,J,K$ and $J \cap K$ are given by

\begin{center}
$0 \rightarrow R(-6) \rightarrow R(-5)^3 \rightarrow R(-3)^3 \rightarrow I.$\\
$0 \rightarrow R(-3) \rightarrow J.$\\
$0 \rightarrow R(-5) \rightarrow R(-3)^2 \rightarrow K.$\\
$0 \rightarrow R(-6) \rightarrow R(-5)^2 \rightarrow J \cap K.$\\
\end{center}

Nevertheless $K$ is not componentwise linear, since $K$ is generated in one degree and has no linear resolution.
\end{ex}

Several examples and Proposition \ref{prequel} suggest the following question.\\

{\bf Question}: Assume $I$ componentwise linear. Does the converse of Theorem \ref{my} hold?% without assumptions on $J \cap K$?

\section{Betti splitting for simplicial complexes}
In this section we present some applications of Theorem \ref{my} to simplicial complexes. For more definitions about simplicial complexes, their properties and the Stanley-Reisner correspondence we refer to \cite[Chapter 1]{hibi} and \cite[Chapter 3]{jonsson} and references over there.

\begin{definition}
\em An {\em abstract simplicial complex} $\Delta$ on $n$ vertices is a collection of subsets of $\{1,\dots,n\}$, called {\em faces}, such that if $F \in \Delta$, $G \subseteq F$, then $G \in \Delta$.
\end{definition} 

Denote a face $A=\{i_1,\dots,i_q\}$ by $A=[i_1,\dots,i_q]$, with $i_1<\dots<i_q$. A {\em facet} is a maximal face of $\Delta$ with respect to the inclusion of sets. Denote by $\F(\Delta)$ the collection of facets of $\Delta$. A simplicial complex is called {\em pure} if all the facets of $\Delta$ have the same cardinality.
\vskip 2mm
The {\em Alexander dual ideal} $I_{\Delta}^*$ of $\Delta$ is defined by 
\vskip 1.5mm
\begin{center}
$I_{\Delta}^*=(x_{\overline{F}}$ $|$ $F \in \F(\Delta))$, where $\overline{F}:=\{1,\dots,n\} \setminus F$ and $x_{\overline{F}}=\prod_{i \in \overline{F}} x_i$.
\end{center}
\vskip 1.5mm
A decomposition $\Delta=\Delta_1 \cup \Delta_2$, such that $\mathcal{F}(\Delta)$ is the disjoint union of $\mathcal{F}(\Delta_1)$ and $\F(\Delta_2)$ induces a decomposition $I_{\Delta}^*=I_{\Delta_1}^*+I_{\Delta_2}^*$. We call $\Delta=\Delta_1 \cup \Delta_2$ a {\em Betti splitting} of $\Delta$ if $I_{\Delta}^*=I_{\Delta_1}^*+I_{\Delta_2}^*$ is a Betti splitting of $I_{\Delta}^*$. Note that all the Alexander dual ideals involved are computed with respect to the vertices of $\Delta$.
\vskip 2mm
%We summarize all the implications in the following diagram.
%$$%\text{facet-splitting} \ar[d]
%\xymatrix{
%\text{vertex-decomposable} \ar[d] \ar[r] & \text{$x_i$-splitting} \ar[dd] \\
%\text{shellable} \ar[d] &  & \\
%\text{constructible} \ar[r] \ar[d] & \text{Betti splitting} & \\
%\text{seq. Cohen-Macaulay} & &
%}
%$$
In the following diagram (see e.g.\cite{jonsson}) we recall the hierarchy of some properties of (possibly non-pure) simplicial complexes.
\begin{equation}\label{diag}\tag{4.1}
\text{vertex decomposable} \rightarrow \text{shellable} \rightarrow \text{constructible} \rightarrow \text{sequentially Cohen-Macaulay}.
\end{equation}

It can be proved (\cite{herwel}, see also \cite[Theorem 8.2.20]{hibi}) that a simplicial complex $\Delta$ is sequentially Cohen-Macaulay if and only if $I_{\Delta}^*$ is componentwise linear. 

All these three properties are defined recursively and for this reason we can apply Theorem \ref{my} to $I_{\Delta}^*$.

In \cite[Theorem 2.3]{ur} Francisco, H\`a and Van Tuyl give conditions on $J,K$ and $J \cap K$ forcing $I=J+K$ to be a Betti splitting of $I$. The splitting given in Theorem \ref{my} is not a consequence of \cite[Theorem 2.3]{ur}, as shown in the following example.

\begin{ex} %By \cite[Proposition 8.2.20]{hibi}, $I$ is componentwise linear.
\em Let $\K$ be a field of characteristic zero and let $I_{\Delta}^* \subseteq \K[x_1,...,x_{12}]$ be the Alexander dual ideal of the simplicial complex $\Delta$ whose set of facets:
\vskip 1.5mm
\begin{center}$\F(\Delta)=\{[1, 2, 3, 4], [1, 2, 3, 12],[3, 4, 6], [3, 4, 5], [4, 5, 6], [3, 5, 6],$\\ $[5, 6, 7],[5, 7, 8],[4, 9],[9, 10],[10, 11],[6, 9], [8, 12]\}.$\end{center}
\vskip 1.5mm
Let $\Delta=\Delta_1 \cup \Delta_2$, with  $\mathcal{F}(\Delta_1)=\{[1, 2, 3, 4],$ $[1, 2, 3, 12],$ $[3, 4, 5], [3, 4, 6], [4, 9]\}$ and $\mathcal{F}(\Delta_2)=\{[3, 5, 6], [4, 5, 6], [5, 6, 7], [5, 7, 8],$ $[6, 9], [8, 12], [10, 11],$ $[9, 10]\}$.

Let $J_i$ be the Alexander dual ideal of $\Delta_i$, for $i=1,2$. Both $\Delta_1$ and $\Delta_2$ are shellable (use {\em Macaulay} \cite{m2}). Then $J_1$ and $J_2$ are componentwise linear. By Theorem \ref{my}, $I_{\Delta}^*=J_1+J_2$ is a Betti splitting of $I$. Nevertheless the assumptions of \cite[Theorem 2.3]{ur} are not satisfied since $\beta_{1,11}(J_1 \cap J_2)>0$ and both $\beta_{1,11}(J_1)$ and $\beta_{1,11}(J_2)$ are not zero.
\end{ex}

In the next result we focus on a special splitting of a monomial ideal $I$. Let $x_i$ be a variable of $R$. Let $J$ be the ideal generated by all monomials of $G(I)$ divided by a $x_i$ and let $K$ be the ideal generated by the remaining monomials of $G(I)$. If $I=J+K$ is a Betti splitting, we call $I=J+K$ a {\em $x_i$-splitting} of $I$.
\vskip 1.5mm
We recover the following known result.

\begin{corollary}\label{vd}\em(\cite[Theorem 2.8, Corollary 2.11]{moradi})
\em If $\Delta$ is a vertex decomposable simplicial complex then there exists $i \in V(\Delta)$ such that $I_{\Delta}^*$ admits $x_i$-splitting. 
\end{corollary}

\begin{proof}
Since $\Delta$ is vertex decomposable, there exists a vertex $i \in V(\Delta)$ such that $\mathrm{del}_{\Delta}(i)$ and $\mathrm{link}_{\Delta}(i)$ are both vertex decomposable. By \cite[Lemma 2.2]{moradi} we have $I_{\Delta}^*=x_iI_{\mathrm{del}_{\Delta}(i}^*+I_{\mathrm{link}_{\Delta}(i)}^*$. A vertex decomposable simplicial complex is sequentially Cohen-Macaulay by Diagram \ref{diag}. Then $x_iI_{\mathrm{del}_{\Delta}(i)}^*$ and $I_{\mathrm{link}_{\Delta}(i)}^*$ are componentwise linear. The statement follows from Theorem \ref{my}.
\end{proof}

In Corollary \ref{vd}, vertex decomposability cannot be replaced by shellability or constructibility, as it is shown in the next example.

\begin{ex}\label{shnobordo}
\em Consider the simplicial complex $\Delta$ whose set of facets is
\vskip 2mm
$\F(\Delta)=\{[2,3,4], [2,4,7],  [1,2,7], [1,6,7], [2,3,5], [1,2,5], [1,2,6], [2,3,6],$

$ [3,5,6], [5,6,7],[4,5,7],[1,4,5],[1,3,4], [1,3,7], [3,5,7], [4,6,7], [4,6,12],$

$[6,11,12], [6,8,12], [8,9,12],[6,8,9],[6,9,10],[6,10,11],[9,10,11],[8,9,11],$

$[4,11,12],[4,8,11],[4,9,12],[4,9,10],[4,8,10], [8,10,12], [10,11,12]\}.$
\vskip 1.5mm
The given order of the facets of $\Delta$ is indeed a shelling, thus $\Delta$ is shellable (use {\em Macaulay} \cite{m2}). By Diagram \ref{diag}, $I_{\Delta}^*$ is componentwise linear. Since $I_{\Delta}^*$ is generated in degree $9$, $I$ has actually a $9$-linear resolution. Consider the splitting $I_{\Delta}^*=x_iJ+K$, for every $1 \leq i \leq 12$. The resolution of $x_iJ$ is {\em not} linear, for every $1 \leq i \leq 12.$ By Proposition \ref{prequel}, $I$ does not admit $x_i$-splitting.
\end{ex}

Let $\Delta$ be a constructible simplicial complex. Although in general $I_{\Delta}^*$ does not admit any $x_i$-splitting, it admits a Betti splitting, as a consequence of Theorem \ref{my}.

\begin{corollary}\label{sh}
Let $\Delta$ be a constructible simplicial complex. Then $I_{\Delta}^*$ admits Betti splitting.
\end{corollary}

\begin{proof} By definition of constructible complexes, there exist two constructible simplicial complexes $\Delta_1$ and $\Delta_2$ such that $\Delta=\Delta_1 \cup \Delta_2$ and $\F(\Delta)$ is the disjoint union of $\F(\Delta_1)$ and $\F(\Delta_2)$. Then $I_{\Delta_1}^*$ and $I_{\Delta_2}^*$ are componentwise linear ideals, by Diagram \ref{diag}. By Theorem \ref{my}, $I_{\Delta}^*=I_{\Delta_1}^*+I_{\Delta_2}^*$ is a Betti splitting of $I_{\Delta}^*$. \end{proof} 

In the pure case the previous result is \cite[Corollary 3.4]{anda}. With the same notations in the proof of Corollary \ref{sh}, for shellable complexes a more precise result holds: $\Delta_2$ consists of a single facet, because in this case $I_{\Delta}^*$ has linear quotients (see \cite{bolo}).\\

%\begin{remark}\label{nosplit}
Notice that Corollary \ref{sh} does not hold for (sequentially) Cohen-Macaulay simplicial complexes, as shown in the following example. We present an ideal $I$ that {\em does not admit Betti splitting at all}, in characteristic different from two. To our knowledge this is the first example in literature of an ideal that does not admit any Betti splitting (see \cite[Example 5.31]{bolo} for a characteristic-free example).
%\end{remark}

\begin{ex}\label{nosplit}
\em Let $\K$ be a field of $\mathrm{char}(\K) \neq 2$. Let $\Delta$ be the following $6$-vertex triangulation with $10$ facets of the real projective plane, due to Reisner \cite{reisner}.

\begin{figure}[!h]\centering
\includegraphics[width=40mm]{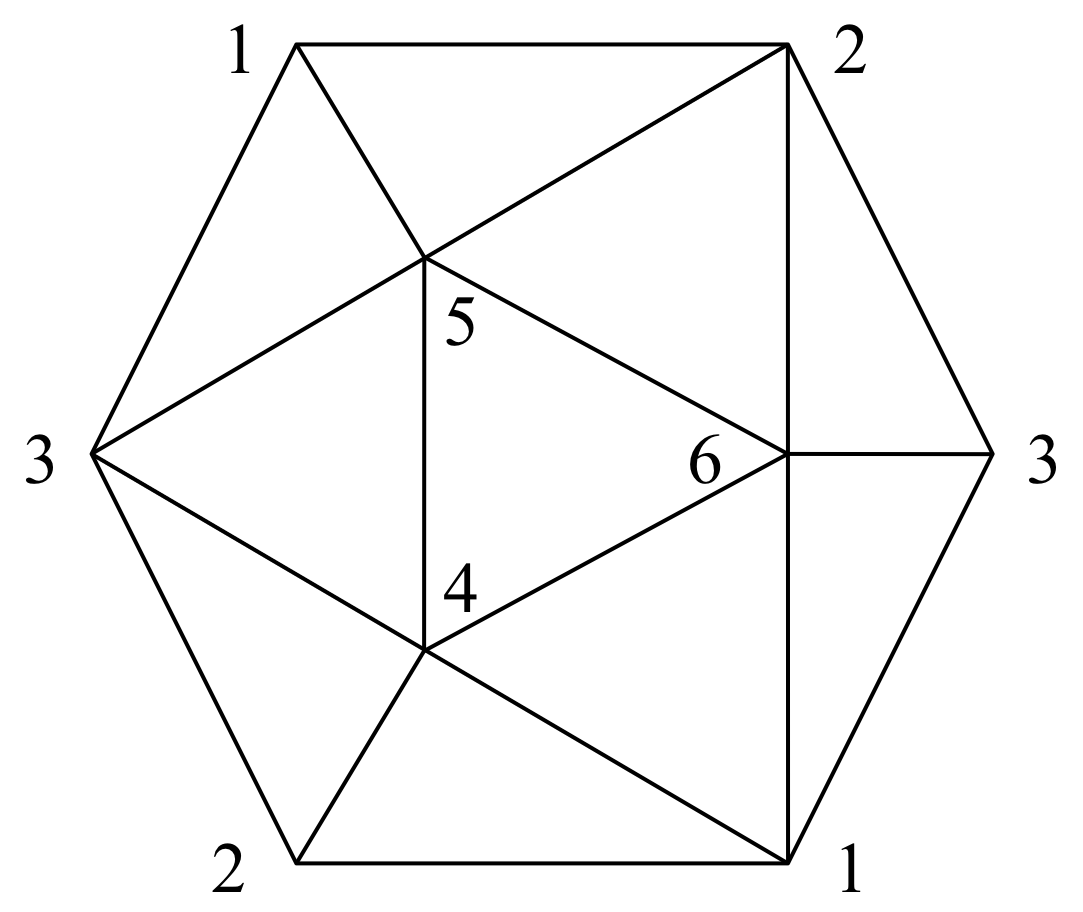}
\caption{A triangulation of real projective plane.}
\end{figure}

%It is a $2$-dimensional manifold without boundary that it is not orientable. It does not admits perfect Morse function in characteristic $\neq 2$, but it admits perfect Morse function in characteristic $2$ (\cite{andr}).

%It is the simplest example of simplicial complex that is Cohen-Macaulay in all fields $\mathbb{K}$ of characteristic $\neq 2$ but that is not Cohen-Macaulay in characteristic $2$ %(it is Buchsbaum in this case). First we consider a field $\K$ of characteristic $\neq 2$. The resolution of $I_{\Delta}^*$ is

The graded Betti numbers of $I_{\Delta}^*$ are given by

$$0 \rightarrow R(-5)^6 \rightarrow R(-4)^{15} \rightarrow R(-3)^{10} \rightarrow I_{\Delta}^*.$$

 Francisco, H\`a and Van Tuyl in \cite{ur} pointed out that $I_{\Delta}^*$ does not admit $x_i$-splitting (in this case the Stanley-Reisner ideal $I_{\Delta}$ and $I_{\Delta}^*$ coincide). Actually $I_{\Delta}^*$ does not admit {\em any} Betti splitting. Using CoCoA \cite{cocoa} and Macaulay \cite{m2}, we checked all the possible $\sum_{i=1}^5\binom{10}{i}=637$ pairs of ideals $J$ and $K$ such that $G(I_{\Delta}^*)$ is the disjoint union of $G(J)$ and $G(K)$. In each case, at least one between $J$ and $K$ has no linear resolution. Then, by Proposition \ref{prequel}, $I_{\Delta}^*=J+K$ is not a Betti splitting.
\end{ex}

%Note that the situation is equal for $1 \leq i \leq 6$.

%\begin{center}
%$0 \rightarrow R(-6) \rightarrow R(-5)^5 \rightarrow R(-4)^5 \rightarrow I_{\cap_{x_i}}^*$\\
%$0 \rightarrow R(-5) \rightarrow R(-4)^5 \rightarrow R(-3)^5 \rightarrow I_{star_{\Delta}(x_i)}^*$\\
%$0 \rightarrow R(-6) \rightarrow R(-4)^5 \rightarrow R(-3)^5 \rightarrow I_{costar_{\Delta}(x_i)}^*$
%&\end{center}

%Then this complex is Cohen-Macaulay (Proposition \ref{CM}). 

%the assumption the ideal $J=(f_1,..,f_{k-1})$ has linear quotients and the ideal $K=(f_k)$ has a linear resolution. Then $J$ and $K$ are componentwise linear. By Theorem \ref{my} the result follows. %$I_{\Delta}^*=(f_1,..,f_k)$ has linear quotients in the given order and then it is componentwise linear.

%Let $\Delta$ be the complex in Example \ref{shnobordo}. Since the resolutions of $I_{\Delta_1}^*$ and $I_{\Delta_2}^*$ are linear, $I_{\Delta}=I_{\Delta_1}^*+I_{\Delta_2}^*$ is a Betti splitting of $I_{\Delta}^*$.\\

%The analogous result can be proved for {\bf constructible} simplicial complexes, not necessarly pure (this result in the pure case is \cite[Corollary 3.4]{anda}).\\ 

\section{The resolution of the ideal of three general fat points in $\mathbb{P}^{n-1}$}\label{fatpoints}
Let $X=\{(P_1,a),(P_2,b),(P_3,c)\}$ be the $0$-dimensional scheme consisting of three general fat points in $\mathbb{P}^{n-1}$, with $n \geq 4$ and $1 \leq a \leq b \leq c$. After  a change of coordinates, we may assume that $P_1=[1:0:...:0]$, $P_2=[0:1:0:...:0]$ and $P_3=[0:0:1:0:...:0]$. Then the defining ideal of $X$ is $$I_{n,a,b,c}=(x_2,...,x_n)^a \cap (x_1,x_3,...,x_n)^b \cap (x_1,x_2,x_4,...,x_n)^c \subseteq R=\K[x_1,...,x_n].$$ In the case of two fat points we denote $I_{n,0,b,c}$ by $I_{n,b,c}$, with $b \leq c$. If $a=b=c$ we denote $I_{n,a,a,a}$ by $I_{n,a}$. %By convention $I_{n,b,c}=0$ if $n \leq 2.$

Francisco proved  in \cite{fran} that the defining ideal of the zero-dimensional schemes of $r \leq n+1$ general fat points in $\mathbb{P}^{n}$ is componentwise linear. In general the ideals $I_{n,a,b,c}$ are not stable (even if $a=b=c=1$). Valla  \cite[Corollary 3.5]{tito} computed the graded Betti numbers of the defining ideal of two general fat points in $\mathbb{P}^n$, $n \geq 2$, by using a Betti splitting argument. We prove a splitting result for $I_{n,a,b,c}$ and, as a consequence, we give a recursive procedure to compute the graded Betti numbers of $I_{n,a,b,c}$ in the case $a \neq c$.

\begin{theorem}\label{splitting}
Let $n \in \N$, $n \geq 4$. Let $X=\{(P_1,a),(P_2,b),(P_3,c)\}$ be the $0$-dimensional scheme defined by three general fat points in $\mathbb{P}^{n-1}$, with $1 \leq a \leq b \leq c$. Assume $c \neq a$. Then $I=I_{n,a,b,c}$ admits $x_1$-splitting.
\end{theorem}

\begin{proof}
Let $J=x_1 I_{n,a,b-1,c-1}$ and $K=(x_3,...,x_n)^b \cap (x_2,x_4,...,x_n)^c$. We show that $I=J+K.$ 

Let $f \in G(I).$ Assume $f \in (x_1).$ Then $f \in (x_1) \cap I.$ Since $ (x_1) \cap (x_1,x_3,...,x_n)^b=x_1(x_1,x_3,...,x_n)^{b-1}$ and $(x_1) \cap (x_1,x_2,x_4,...,x_n)^c=x_1(x_1,x_2,x_4,...,x_n)^{c-1}$, it follows $$(x_1) \cap I=x_1(x_2,...,x_n)^a \cap x_1(x_1,x_3,...,x_n)^{b-1} \cap x_1(x_1,x_2,x_4,...,x_n)^{c-1}=x_1I_{n,a,b-1,c-1}=J.$$ 

Let $f \notin (x_1),$ then $f \in (x_2,..,x_n)^a \cap (x_3,..,x_n)^b \cap (x_2,x_4,..,x_n)^c$. Since $(x_3,..,x_n) \subseteq (x_2,x_3,..,x_n)$ and $a \leq b,$ then $f \in (x_3,..,x_n)^b \cap (x_2,x_4,..,x_n)^c=K$. For the other inclusion, note that $J=(x_1) \cap I \subseteq I$ and $K=(x_3,..,x_n)^b \cap (x_2,x_4,..,x_n)^c=(x_2,..,x_n)^a \cap (x_3,..,x_n)^b \cap (x_2,x_4,..,x_n)^c \subseteq I$. 

We prove now that $G(I)=G(J) \cup G(K)$. Let $g \in G(I).$ Then $g \in J$ or $g \in K$. Assume $g \in J.$ If $g \notin G(J)$, there would be $m \in G(J)$ such that $g \in (m).$ Since $J \subseteq I$ and $g \in G(I)$, this is a contradiction. The proof for $K$ is the same, then the first inclusion is clear.

For the reverse inclusion, we first prove that $K \subseteq \M I_{n,a,b-1,c-1}.$ Let $h \in G(K)$ be a monomial. Assume first that there is a variable $x_i,$ with $x_i|h$ and $4 \leq i \leq n$. Then $$h \in x_i[(x_3,..,x_n)^{b-1} \cap (x_2,x_4,..,x_n)^{c-1}] \subseteq \M[(x_3,..,x_n)^{b-1} \cap (x_2,x_4,..,x_n)^{c-1}].$$ Since $a \leq c-1$ and $(x_2,x_4,..,x_n) \subseteq (x_2,..,x_n)$ hence $$h \in \M[(x_2,..,x_n)^a \cap (x_3,..,x_n)^{b-1} \cap (x_2,x_4,..,x_n)^{c-1}] \subseteq \M I_{n,a,b-1,c-1}.$$ Otherwise $h=x_2^cx_3^b=x_2x_3(x_2^{b-1}x_3^{c-1}) \in \M I_{n,a,b-1,c-1}.$

Let $g \in G(J)$. By definition of $J$, there is $s \in G(I_{n,a,b-1,c-1})$ such that $g=x_1s.$ If $g \notin G(I)$, then there are $h \in G(I)$ and a monomial $r \in R$ such that $g=x_1s=hr$. Since $G(I) \subseteq G(J) \cup G(K)$ and $g \in G(J),$ then $h \in G(K).$ Hence $x_1|r$ and $s=r_1h,$ where $r_1=\frac{r}{x_1}.$ This is a contradiction, because $h \in \M I_{n,a,b-1,c-1}$. 

Let $g \in G(K)$. If $g \notin G(I)$, then there is $m \in G(I)$ such that $g \in (m)$. Since $g \in G(K)$, then $m \in J=(x_1) \cap I$. Hence $x_1|g,$ a contradiction. 

Clearly one has $G(J) \cap G(K)=\emptyset$. By \cite[Theorem 4.6]{fran}, $J$ and $K$ are componentwise linear. In view of Theorem \ref{my}, we conclude that $I=J+K$ is a Betti splitting of $I$.
\end{proof}

Notice that, in general, the splitting of Theorem \ref{splitting} is not a particular case of \cite[Theorem 2.3]{ur} (see for instance the case $n=4$, $a=b=1$, $c=2$).

In the next corollary we compute explicitly the graded Betti numbers of $I_{n,a,b,c}$ in the case $a+b \leq c$ by a recursive procedure.

\begin{corollary}\label{finalcor}
Let $n \in \N$, $n \geq 4$. Let $X=\{(P_1,a),(P_2,b),(P_3,c)\}$ be the $0$-dimensional scheme consisting of three general fat points in $\mathbb{P}^{n-1}$, with $1 \leq a \leq b \leq c$ and $I=I_{n,a,b,c}$. Assume $a+b \leq c$. Then $$\beta_{i,i+c}(I)=\beta_{i,i+c-b}(I_{n,a,c-b})+\sum_{r=0}^{b-1}[\beta_{i,i+c-r}(I_{n-1,b-r,c-r})+\beta_{i-1,i+c-r-1}(I_{n-1,b-r,c-r})].$$ $$\beta_{i,j}(I)=\begin{cases} \binom{n-2}{i}[\binom{n-3+c+a-j+i}{n-3}+\binom{n-3+c+b-j+i}{n-3}] & \text{   if  } c+1+i \leq j \leq a+c+i. \\ \binom{n-2}{i}\binom{n-3+c+b-j+i}{n-3} & \text{   if  } a+c+1+i \leq j \leq b+c+i. \\ 0 & \text{   if  } j \geq b+c+1+i. \end{cases}$$
\end{corollary}

\begin{proof}
Since $a+b \leq c$, the assumptions of Theorem \ref{splitting} are fullfilled, then $I_{n,a,b,c}$ admits $x_1$-splitting. Let $J$ and $K$ be as in the proof of Theorem \ref{splitting}. Clearly $J \cap K=x_1K$. Note that $\beta_{i,j}(J)=\beta_{i,j-1}(I_{n,a,b-1,c-1})$ and $\beta_{i-1,j}(J \cap K)=\beta_{i-1,j-1}(K),$ for each $i,j \geq 0$. We remark that, after a relabeling of the variables, $K$ is the ideal of two general fat points in $\mathbb{P}^{n-2}$, i.e. $K=I_{n-1,b,c}$. Then $$\beta_{i,j}(I)=\beta_{i,j-1}(I_{n,a,b-1,c-1})+\beta_{i,j}(I_{n-1,b,c})+\beta_{i-1,j-1}(I_{n-1,b,c}).$$ Since $c-r \neq a$, for $0 \leq r \leq c-a-1$ and $b-1 \leq c-a-1$, we can apply the same argument recursively to $I_{n,a,b-r,c-r}$, for $0 \leq r \leq b-1$, to get \begin{equation}\label{rec} \beta_{i,j}(I)=\beta_{i,j-b}(I_{n,a,c-b})+\sum_{r=0}^{b-1}[\beta_{i,j-r}(I_{n-1,b-r,c-r})+\beta_{i-1,j-1-r}(I_{n-1,b-r,c-r})].\end{equation} The result follows by \cite[Corollary 3.5]{tito} and by the relation $\sum_{r=h}^s\binom{r}{c}=\binom{s+1}{c+1}-\binom{h}{c+1}$.
\end{proof}

Theorem \ref{splitting} allows us to apply a recursive procedure for computing the Betti numbers of $I= I_{n,a,b,c}$ also in the case $a+b>c$ and $c \neq a. $ This formula has the limit that the Betti numbers of $I_{n,k}$, $k \in \N$, could be involved. These ideals are studied in \cite{fat}. An explicit formula for the graded Betti numbers of $I_{n,k}$ is given only for $k=2$ \cite[Proposition 3.2]{fran} and $k=3$ \cite[Proposition 3.3]{fran}. %In \cite{fran}, Francisco suggest to use a result of Gasharov, Hibi and Peeva \cite{peeva} to compute the resolution of $I_{n,a}.$  

\begin{corollary}\label{finalcor2}
Let $n \in \N$, $n \geq 4$. Let $X=\{(P_1,a),(P_2,b),(P_3,c)\}$ be the $0$-dimensional scheme consisting of three general fat points in $\mathbb{P}^{n-1}$, with $1 \leq a \leq b \leq c$, $a+b>c$ and $c \neq a$. Let $I=I_{n,a,b,c}$ be the defining ideal of $X$. Set $k=a+b-c$, \begin{center}$B_{n,a,b,c}^{i,j}:=\binom{n-3+c+a-j+i}{n-3}+\binom{n-3+c+b-j+i}{n-3}-2\binom{n-3+a+b-j+i}{n-3},$ for $i,j \in \N,$\end{center} and \begin{center}$\gamma_{s,t}^{n,i}:=\beta_{i,i+t}(I_{n-1,s,t})+\beta_{i-1,i+t-1}(I_{n-1,s,t}),$ for $s,t \in \N$.\end{center} Then
$$
\beta_{i,i+c}(I)=\beta_{i,i+k}(I_{n,k}) +\sum_{r=0}^{c-a-1}\gamma_{b-r,c-r}^{n,i}+\sum_{r=0}^{c-b-1}\gamma_{a-r,a-r}^{n,i}.
$$
$$\beta_{i,j}(I)=\begin{cases}
\beta_{i,j+k-c}(I_{n,k})+\binom{n-2}{i}B_{n,a,b,c}^{i,j} & \text{ if  } {c+1+i \leq j \leq a+b+i.}\\
\binom{n-2}{i}[\binom{n-3+c+a-j+i}{n-3}+\binom{n-3+c+b-j+i}{n-3}] & \text{ if  } {a+b+1+i \leq j \leq a+c+i.}\\
\binom{n-2}{i}\binom{n-3+c+b-j+i}{n-3} & \text{ if  } {a+c+1+i \leq j \leq b+c+i.}\\
0 & \text{   if  } {j \geq c+b+1+i.}\end{cases}$$
\end{corollary}

\begin{proof}
Note that $c<a+b \leq 2c$ and $c \geq 2$. By Theorem \ref{splitting}, $I_{n,a,b,c}$ admits $x_1$-splitting. The main difference with Corollary \ref{finalcor} is that $c-r \neq a$ for $0 \leq r \leq c-a$ but $c-a<b$. Then, following the proof of Corollary \ref{finalcor} and using Equation (\ref{rec}) we get $$\beta_{i,j}(I)=\beta_{i,j+a-c}(I_{n,a,a+b-c,a})+\sum_{r=0}^{c-a-1}[\beta_{i,j-r}(I_{n-1,b-r,c-r})+\beta_{i-1,j-1-r}(I_{n-1,b-r,c-r})].$$ We focus our attention only on the first term $\beta_{i,j+a-c}(I_{n,a,a+b-c,a})$ of the equation. Now one has $a+b-c \leq a$. If $b=c$ the claim follows. If $b<c$, then $a+b-c \neq a$ and the assumptions of Theorem \ref{splitting} are satisfied. Since we consider the new order of the ideals of the intersection given by multiplicities, then $I_{n,a,a+b-c,a}$ admits $x_2$-splitting. One has $a-r \neq a+b-c$ for $0 \leq r \leq c-b-1$. Using Equation (\ref{rec}) and $k=a+b-c$, we get $$\beta_{i,j+a-c}(I_{n,k,a,a})=\beta_{i,j+k-c}(I_{n,k})+\sum_{r=0}^{c-b-1}[\beta_{i,j+a-c-r}(I_{n-1,a-r,a-r})+\beta_{i-1,j+a-c-r-1}(I_{n-1,a-r,a-r})].$$ By \cite[Proposition 3.1]{fran}, the ideal $I_{n,k}$ is generated in degree at most $2k$. The statement follows \cite[Corollary 3.5]{tito}.
\end{proof}

We prove that in Theorem \ref{splitting} the assumption $c \neq a$ is essential.

\begin{ex}\label{nosplit}
\em Consider three double points in $\mathbb{P}^3$. The defining ideal $I=I_{4,2}$ admits the decomposition $I=J+K$ of Theorem \ref{splitting}, where $J=x_1I_{4,2,1,1}$ and $K=(x_3,x_4)^2 \cap (x_2,x_4)^2$. Unfortunately $G(I) \neq G(J) \cup G(K)$, since we have $x_1x_4^2 \in G(J)$ that is {\em not} a minimal generator of $I$. The same problem arise if we choose $x_2$ or $x_3$. 

It can be proved \cite{bolo} that, in general, $I_{n,a}$ admits $x_n$-splitting. More precisely, we have $I_{n,a}=x_nI_{n,a-1}+I_{n-1,a},$ with $G(I_{n,a})=G(x_nI_{n,a-1}) \cup G(I_{n-1,a})$ and $G(x_nI_{n,a-1}) \cap G(I_{n-1,a})=\emptyset$, but we are not able to take advantage from this decomposition, since the resolutions of both $I_{n,a-1}$ and $I_{n-1,a}$ are in general unknown. 
\end{ex}

\bigskip
\noindent
Davide Bolognini\\
Dipartimento di Matematica\\
Universit{\`a} di Genova\\
Via Dodecaneso 35, 16146 Genova, Italy\\
e-mail: {\tt bolognin@dima.unige.it}; {\tt  davide.bolognini@yahoo.it}
\end{document}